\documentclass[a4paper,10pt]{article}
\usepackage{amsfonts}
\usepackage{amssymb,amsthm}
\usepackage{amsmath,comment}

\newcommand{\R}{\mathbb{R}}
\newcommand{\N}{\mathbb{N}}

\newcommand{\T}{\mathcal{T}}

\newcommand{\Pa}{\mathcal{P}}

\newtheorem{theorem}{Theorem}
\newtheorem{lemma}[theorem]{Lemma}
\newtheorem{proposition}[theorem]{Proposition}
\newtheorem{corollary}[theorem]{Corollary}

\begin{document}
\title{On the minimization of Dirichlet eigenvalues}
\author{ { M. van den Berg \thanks{Partially supported by the London Mathematical Society, Grant 41217 and by The Leverhulme Trust,
International Network Grant \emph{Laplacians, Random Walks, Bose
Gas, Quantum Spin Systems}. The author wishes to thank Brian
Davies and Dorin Bucur for helpful
discussions.}} \\
School of Mathematics, University of Bristol\\
University Walk, Bristol BS8 1TW\\
United Kingdom\\
\texttt{mamvdb@bristol.ac.uk}}
\date{14 October 2014}\maketitle

\vskip 3truecm \indent

\begin{abstract}\noindent Results are obtained for two
minimization problems: \begin{equation*}I_k(c)=\inf
\{\lambda_k(\Omega): \Omega\ \textup{open, convex in}\ \R^m,\
\T(\Omega)= c \},\end{equation*} and
\begin{equation*}J_k(c)=\inf\{\lambda_k(\Omega):
\Omega\ \textup{quasi-open in}\ \R^m, |\Omega|\le 1, \mathcal
P(\Omega)\le c \},\end{equation*} where $c>0$, $\lambda_k(\Omega)$
is the $k$'th eigenvalue of the Dirichlet Laplacian acting in
$L^2(\Omega)$, $|\Omega|$ denotes the Lebesgue measure of
$\Omega$, $\mathcal P(\Omega)$ denotes the perimeter of $\Omega$,
and where $\T$ is in a suitable collection of functions. The
latter includes the perimeter of $\Omega$ and the moment of
inertia of $\Omega$ with respect to its centre of mass.
\end{abstract}



\mbox{}\newpage

\section{Introduction\label{sec1}}
Let $\Omega$ be an open set in Euclidean space $\R^m \; (
m=2,3,\cdots)$, with boundary $\partial \Omega$, and let
$-\Delta_{\Omega}$ be the Dirichlet Laplacian acting in
$L^2(\Omega)$. It is well known that if $\Omega$ has finite
Lebesgue measure $|\Omega|$ then $-\Delta_{\Omega}$ has compact
resolvent, and the spectrum of $-\Delta_{\Omega}$ is discrete and
consists of eigenvalues $\lambda_1(\Omega)\le\lambda_2(\Omega)\le
\cdots$ with $\lambda_j(\Omega)\rightarrow\infty$ as
$j\rightarrow\infty$. The Faber-Krahn inequality (Theorem 3.2.1 in
\cite{H}) asserts that if $c>0$ then
\begin{equation*}
\inf \{\lambda_1(\Omega) :\Omega\ \textup{open in}\ \R^m ,\
|\Omega| = c \}
\end{equation*}
is attained for a ball with Lebesgue measure $c$. The Krahn-Szeg\"o
inequality (Theorem 4.1.1 in \cite{H}) asserts that if $c>0$ then
\begin{equation*}
\inf \{\lambda_2(\Omega) :\Omega\ \textup{open in}\ \R^m ,\
|\Omega| = c \}
\end{equation*}
is attained for two disjoint balls each with Lebesgue measure $c/2.$ For
higher Dirichlet eigenvalues ($k>2$) it is not known whether the
variational problem

\begin{equation}\label{e2a}\inf \{\lambda_k(\Omega) :\Omega\
\textup{open in}\ \R^m ,\ |\Omega| = c \}
 \end{equation}
 has a
minimizer. However, it has been shown that if the collection of
open sets in \eqref{e2a} is enlarged to the quasi-open sets then
the variational problem
 \begin{equation}\label{e3}
M_k(c)=\inf \{\lambda_k(\Omega) :\Omega\ \textup{quasi-open in}\
\R^m ,\ |\Omega| = c \}
\end{equation} has a bounded minimizer \cite{B,MP} with finite perimeter
\cite{B}. Even though the class of quasi-open sets with measure
$c$ is much larger than the class of open sets with measure $c,$
the infima under \eqref{e2a} and \eqref{e3} are equal.

Few facts are known about these minimizers. E. Oudet has shown
that the ball is not a minimizer of \eqref{e2a} for $k=3, m=3$.
Furthermore the disc is a local minimum of the functional under
\eqref{e2a} for $k=3,m=2$ \cite{H}, and any minimizer of
\eqref{e2a} or \eqref{e3} for $m=2, k=3$, $m=3, k=3$, and $m=3,
k=4$ is connected \cite{WK,vdBI}. An upper bound for the number of
components of a minimizer of \eqref{e2a} (or \eqref{e3}) has been
obtained in Theorems 1 and 2 of \cite{vdBI} in terms of $k$ and
$m$.

Minimization problems for Dirichlet eigenvalues with other
constraints such as torsional rigidity or perimeter have been
investigated in \cite{K1,K2} and \cite{BH2,DPV} respectively. In
\cite {DPV} it was shown that if $m=2,3,\cdots, k\in \N$, and if
$\Pa(\Omega)$ denotes the perimeter of $\Omega$ then
\begin{equation}\label{e4}
P_k(c)=\inf \{\lambda_k(\Omega) :\Omega\ \textup{open in}\ \R^m ,\
\Pa(\Omega) = c,\ |\Omega|<\infty \}
\end{equation}
has a minimizer with a regular boundary, and that any minimizer is
connected. The situation is very simple for $m=2$: taking the
convex envelope of a component of a planar open set decreases both
its perimeter and all of its Dirichlet eigenvalues. It follows
that if $m=2$ then any minimizer is convex and has diameter
bounded by $c/2$. See, for example, Theorem 4 in \cite{vdBI}.
Further progress was made by Bucur and Freitas \cite{BF} who
proved that if $m=2,$ and if $(\Omega_k^*)_{k\in \N}$ is a
sequence of minimizers of \eqref{e4} for $k\in \N$ respectively
then there exists a sequence of translates of these minimizers
again denoted by $(\Omega_k^*)_{k\in \N}$ such that
$\Omega_k^*\rightarrow \frac{c}{2\pi}D$ as $k\rightarrow \infty,$
where $\frac{c}{2\pi}D$ is a homothety of a disc $D$ with radius
$1$ by a factor $\frac{c}{2\pi}$. The convergence is with respect
to the Hausdorff metric. As these authors point out in \cite{BF} it
is not known whether the minimizers of \eqref{e4} are convex for
$m>2$ or whether their diameters are bounded uniformly and
independently of $k$, see \cite{B,MP}.

In this paper we consider a class of constraints, which includes
perimeter and moment of inertia, under the additional constraint
of convexity. Let
\begin{equation}\label{e5}
I_k(c)=\inf \{\lambda_k(\Omega) :\Omega\ \textup{open, convex in}\
\R^m ,\ \T(\Omega) = c \},
\end{equation}
where $\T$ satisfies the following hypotheses.
\begin{enumerate}
\item[(a)]$\T$ is a set function defined on the open,
convex sets in $\R^m$ which is (i) invariant under isometries,
(ii) monotone, i.e. $\Omega_1,\Omega_2$ convex with
$\Omega_1\subset\Omega_2$ implies $\T(\Omega_1)\le \T(\Omega_2)$,
(iii) non-negative and $\T(\Omega)=0$ if and only if
$\Omega=\emptyset$.
\item[(b)]There exists $\tau>0$ such that if $\alpha>0$, and if $\Omega$ is open and convex
then $\T(\alpha\Omega)=\alpha^{\tau}\T(\Omega)$.
\item[(c1)]$T^*$ defined by
\begin{equation}\label{e6}T^*=\inf \{\T(\Omega) :\Omega\ \textup{open, convex in}\ \R^m
,\ |\Omega|=1\}
\end{equation} is strictly positive.
\item[(c2)]
There exists an open, convex set $D$ with $|D|=1$ which is unique
up to isometries such that
\begin{equation}\label{e6a}
\T(D)=T^*.
\end{equation}
\item[(d)]There exist constants $K<\infty$ and $t>1/{\tau}$ such that if  $\Omega$ is open, bounded and convex
then \begin{equation}\label{e7} \textup{diam}(\Omega)\le
K\T(\Omega)^t|\Omega|^{(1-t\tau)/m}.
\end{equation}
\end{enumerate}

We remark that (a) and (c2) imply (c1). Our first result is the
following.
\begin{theorem}\label{the2} Let $m=2,3,\cdots$ and let
$k=1,2,\cdots$.
\begin{enumerate}
\item[\textup{(i)}]If $\T$ satisfies \textup{(a), (b) and (c1)} then variational problem \eqref{e5} has a
minimizer.
\item[\textup{(ii)}]If $\T$ satisfies \textup{(a), (b), (c2)} and $\textup{(d)},$ and if $(\Omega_k^*)_{k\in \N}$ are minimizers of \eqref{e5} for $k\in \N$ respectively then
there exists a sequence of isometries of these minimizers again
denoted by $(\Omega_k^*)_{k\in \N}$ such that
\begin{equation}\label{e6b}
\Omega_k^*\rightarrow \left(\frac{c}{\T(D)}\right)^{1/{\tau}}D,
\end{equation}
where the convergence is with respect to both the Hausdorff metric and
the complementary Hausdorff metric.
\end{enumerate}
\end{theorem}
In \cite{HO} the authors study variational problem \eqref{e5} in the case where
$\T$ is Lebesgue measure, and obtain properties of
minimizers. Here we note that the Lebesgue measure constraint
satisfies (a), (b) and (c1). Theorem \ref{the2}(i) confirms the
existence of a minimizer in that case. However, this constraint
does not satisfy (c2) nor does it satisfy (d). So we do not obtain
any information about the asymptotic behaviour of these minimizers
for large $k$.

We remark that if $\T_1$ and $\T_2$ are constraints which satisfy
(a), (b) and (d) with constants $\tau_1, t_1, K_1$ and $\tau_2,
t_2, K_2$ respectively and if there exists a convex set $D$ such that (c2) holds for both $\T_1$ and $\T_2$ then $\T_1\T_2$ defined by
$(\T_1\T_2)\Omega=\T_1(\Omega)\T_2(\Omega)$ satisfies (a), (b)
with $\tau = \tau_1 + \tau_2$, (d) with $t=\frac{t_1t_2}{t_1+t_2}$
and $K=\max\{K_1,K_2\}$, and (c2) with $D$.

We defer the proof of Theorem \ref{the2} to
Section \ref{sec2}. There we also present and prove some of its corollaries.

Our second result is an interpolation between the minimization of
the $k$'th eigenvalue with a Lebesgue measure constraint, and of
the $k$'th eigenvalue with a perimeter constraint. Since existence
of a minimizer of the former has been shown for quasi-open sets,
we define
\begin{equation}\label{b1}
J_k(c)=\inf\{\lambda_k(\Omega): \Omega\ \textup{quasi-open in}\
\R^m, |\Omega|\le 1, \mathcal P(\Omega)\le c \}.
\end{equation}
We denote by $\mathfrak{M}_k$ the collection of minimizers of
$M_k(1)$, and by $\mathfrak{P}_k$ the collection of minimizers of
$P_k(1)$ respectively. These collections are non-empty by the
results of \cite{B,MP} and \cite{DPV} respectively. Let
\begin{equation*}
\pi_k=\inf\{|\Omega|:\Omega \in \ \mathfrak{P}_k\},
\end{equation*}
and
\begin{equation*}
\mu_k=\inf\{\Pa(\Omega):\Omega \in \ \mathfrak{M}_k\}.
\end{equation*}
We also denote by $\omega_m$ the Lebesgue measure of the ball in $\R^m$ with radius $1$.
\begin{theorem}\label{the3} Let $m=2,3,\cdots$, and let $k=1,2,\cdots$.
\begin{enumerate}
\item[\textup{(i)}]$c\mapsto J_k(c)$ is monotonically decreasing, continuous on
$\R^+$, and
\begin{equation}\label{b4}
c_1^{2/(m-1)}J_k(c_1)\le c_2^{2/(m-1)}J_k(c_2),\ \ 0<c_1\le
c_2<\infty.
\end{equation}
\item[\textup{(ii)}] If $ c>\mu_k$ then there exists $\Omega^*\in\mathfrak{M}_k$ which is a minimizer of \eqref{b1}.
If $ c\ge\mu_k$ then
\begin{equation*}
J_k(c)=M_k(1).
\end{equation*}
\item[\textup{(iii)}] If $\ 0<c<\pi_k^{-(m-1)/m}$ then there exists $\Omega_*\in\mathfrak{P}_k$ such that $c^{1/(m-1)}\Omega_*$ is a minimizer of \eqref{b1}.
If $\ 0<c\le\pi_k^{-(m-1)/m}$ then
\begin{equation*}
J_k(c)=c^{-2/(m-1)}P_k(1).
\end{equation*}
\item[\textup{(iv)}]\begin{equation*}
\mu_k\ge m\omega_m^{1/m}.
\end{equation*}
\item[\textup{(v)}]\begin{equation*}
\pi_k\le
m^{-m/(m-1)}\omega_m^{-1/(m-1)}.
\end{equation*}
\item[\textup{(vi)}]\begin{equation*}
\pi_k\ge (2m)^{-m/(m-1)}(m+2)^{-m/2}\omega_m^{-1}.
\end{equation*}
\end{enumerate}
\end{theorem}
We do not have a proof of existence of a minimizer of
\eqref{b1} for $\pi_k^{-(m-1)/m}<c\le \mu_k$. However, the proof
of Theorem \ref{the3} does not rely on that existence. We defer
the proof of Theorem \ref{the3} to Section \ref{sec3}.

\section{Proof of Theorem \ref{the2} \label{sec2}}

Throughout we will denote the inradius of a set $A$ by
\begin{equation*}
\rho(A)=\sup\{\rho>0:x\in A, B(x;\rho)\subset A\},
\end{equation*}
where $B(x;\rho)$ is the open ball with centre $x$ and radius
$\rho$. The following will be used in the proof of Theorem
\ref{the2}.
\begin{lemma}\label{conv}
If $\Omega$ is an open, convex set in $\R^m$ with inradius
$\rho(\Omega)$ and with finite Lebesgue measure $|\Omega|$ then
$\Omega$ is bounded, and
\begin{equation}\label{e15a}
\textup{diam}(\Omega)\le
2m\omega_{m-1}^{-1}\rho(\Omega)^{1-m}|\Omega|.
\end{equation}
If $\Omega$ is an open, convex set in $\R^m$ with finite Lebesgue
measure $|\Omega|,$ then
\begin{equation}\label{e15b}
\rho(\Omega)\ge2^{m-1}(m\omega_m)^{-1}\textup{diam}(\Omega)^{1-m}|\Omega|.
\end{equation}
\end{lemma}
\begin{proof}
Let $d$ be any point of the boundary of an open, convex set
$\Omega$ with Lebesgue measure $|\Omega|$, and let $0$ be the
centre of an open ball with radius $\rho(\Omega)$ in $\Omega$. Let
$\Pi$ be the $(m-1)$-dimensional plane through $0$ perpendicular
to the straight line segment $[0,d]$. The disc $\Pi \cap
B(0;\rho(\Omega))$ has $(m-1)$-dimensional Lebesgue measure
$\omega_{m-1}\rho(\Omega)^{m-1}$. By convexity we have that the
cone with base $\Pi \cap B(0;\rho(\Omega))$ and vertex $d$ is
contained in $\Omega$. Since the Lebesgue measure of that cone is
given by $m^{-1}\omega_{m-1}|d|\rho(\Omega)^{m-1},$ we conclude
that
\begin{equation}\label{e15}
|\Omega|\ge m^{-1}\omega_{m-1}|d|\rho(\Omega)^{m-1}.
\end{equation}
This implies that $\Omega$ is bounded. Then by \eqref{e15},
\begin{align*}
\textup{diam}(\Omega)&=
\sup\{|d_1-d_2|:d_1\in\partial\Omega,d_2\in\partial\Omega\}\nonumber
\\ &\le 2\sup\{|d|:d\in
\partial\Omega\}\nonumber \\
&\le2m\omega_{m-1}^{-1}\rho(\Omega)^{1-m}|\Omega|.
\end{align*}
This proves \eqref{e15a}.

To prove \eqref{e15b}, we have by \cite{O} that
\begin{equation}\label{e15c}
\rho(\Omega)\ge \Pa(\Omega)^{-1}|\Omega|.
\end{equation}
Since $\Omega$ is convex and contained in a ball with radius
$\frac{1}{2}\textup{diam}(\Omega)$ we have by Proposition 2.4.3
(i) in \cite{BB} that
 \begin{equation}\label{e15d}\Pa(\Omega)\le
m\omega_m\left(\frac{1}{2}\textup{diam}(\Omega)\right)^{m-1}.
\end{equation}
Inequality \eqref{e15b} follows from \eqref{e15c} and \eqref{e15d}.
\end{proof}

Below we obtain estimates for $|\T(A)-\T(B)|$ and
$|\lambda_k(A)-\lambda_k(B)|$ for two bounded convex sets $A$ and
$B$ in terms of their Hausdorff distance $d^H(A,B)$.
\begin{lemma}\label{Haus}
If $A$ and $B$ are two open bounded convex sets in $\R^m$, if $\T$
is a set function satisfying hypotheses \textup{(a)} and
\textup{(b)}, and if $\epsilon:=d^H(A,B)\le\rho(A)/2$ then
\begin{equation}\label{a151}
|\T(A)-\T(B)|\le \frac{2\tau3^{\tau}\epsilon}{\rho(A)}\T(A),
\end{equation}
and
\begin{equation}\label{a152}
|\lambda_k(A)-\lambda_k(B)|\le
\frac{16\epsilon}{\rho(A)}\lambda_k(A).
\end{equation}
\end{lemma}
\begin{proof}
Define the $\epsilon$-neighbourhood of a set by
$\Omega^{\epsilon}=\{x\in\R^m:
\textup{dist}(x,\Omega)<\epsilon\}$. Then $d^H(A,B)=\epsilon$
implies $B\subset A^{\epsilon}$ and $A\subset B^{\epsilon}$. But
since $A$ is convex $A^{\epsilon}\subset
(1+\frac{\epsilon}{\rho(A)})A$, where the latter homothety is with
respect to the centre of an inball. Then by monotonicity and
scaling we have that for $\epsilon\le \rho(A)/2$,
\begin{align}\label{a153}
\T(B)-\T(A)&\le \T(A^{\epsilon})-\T(A)\nonumber \\ &\le
\T((1+\frac{\epsilon}{\rho(A)})A)-\T(A)\nonumber \\ &
=((1+\frac{\epsilon}{\rho(A)})^{\tau}-1)\T(A)\nonumber \\ &\le
\frac{\tau3^{\tau}\epsilon}{2^{\tau}\rho(A)}\T(A).
\end{align}
Reversing the roles of $A$ and $B$ we obtain by using $\rho(B)\ge
\rho(A)-\epsilon$, and \eqref{a153} that for $\epsilon\le
\rho(A)/2$,
\begin{align}\label{a154}
\T(A)-\T(B)&
\le((1+\frac{\epsilon}{\rho(B)})^{\tau}-1)\T(B)\nonumber \\ &\le
((1+\frac{\epsilon}{\rho(A)-\epsilon})^{\tau}-1)(\T(B)-\T(A)+\T(A))\nonumber \\ &
\le((1+\frac{\epsilon}{\rho(A)-\epsilon})^{\tau}-1)(1+\frac{\epsilon}{\rho(A)})^{\tau}\T(A)\nonumber \\ &\le \frac{2\tau 3^{\tau}\epsilon}{\rho(A)}\T(A).
\end{align}
Inequality \eqref{a151} follows by \eqref{a153} and \eqref{a154}.

To prove \eqref{a152} we use the scaling and monotonicity of the Dirichlet eigenvalues to obtain that
\begin{align}\label{a155}
\lambda_k(A)-\lambda_k(B)&\le \lambda_k(A)-\lambda_k(A^{\epsilon})\nonumber \\ &\le \lambda_k(A)-\lambda_k((1+\frac{\epsilon}{\rho(A)})A)\nonumber \\ &
=(1-(1+\frac{\epsilon}{\rho(A)})^{-2})\lambda_k(A)\nonumber \\ &\le \frac{2\epsilon}{\rho(A)}\lambda_k(A).
\end{align}
Reversing the roles of $A$ and $B$ we obtain that
\begin{equation}\label{a156}
\lambda_k(B)-\lambda_k(A)\le\frac{2\epsilon}{\rho(B)}\lambda_k(B).
\end{equation}
Since $A\subset B^{\epsilon}\subset(1+\frac{\epsilon}{\rho(B)})B$
we have that $(1+\frac{\epsilon}{\rho(B)})^{-1}A\subset B$. Hence
\begin{equation}\label{a157}
\lambda_k(B)\le (1+\frac{\epsilon}{\rho(B)})^{2}\lambda_k(A).
\end{equation}
So by \eqref{a156} and \eqref{a157} we obtain by using $\rho(B)\ge
\rho(A)-\epsilon$ and $\epsilon\le \rho(A)/2$ that
\begin{align}\label{a158}
\lambda_k(B)-\lambda_k(A)&\le\frac{2\epsilon}{\rho(B)}(1+\frac{\epsilon}{\rho(B)})^{2}\lambda_k(A)\nonumber
\\ & \le \frac{16\epsilon}{\rho(A)}\lambda_k(A).
\end{align}
Inequality \eqref{a152} follows from \eqref{a155} and
\eqref{a158}.
\end{proof}

In order to prove Theorem \ref{the2}(i) we let $c>0$, fix $k\in
\N$, and let $(\Omega_{k,n})_{n\in \N}$ be a minimizing sequence
of \eqref{e5}. We first show that $\textup{diam}(\Omega_{k,n})$ is
uniformly bounded in $n$. It follows from \eqref{e6} and
hypothesis (b) that for any convex $\Omega$ with finite Lebesgue
measure
\begin{equation}\label{e8}
T^*\le\frac{\T(\Omega)}{|\Omega|^{\tau/m}}.
\end{equation} Since $\T(\Omega_{k,n})=c$ we have that
\begin{equation}\label{e9}
|\Omega_{k,n}|\le \left(\frac{c}{T^*}\right)^{m/{\tau}}.
\end{equation} By hypothesis (c1), $T^*>0$, and so the left hand
side of \eqref{e9} is uniformly bounded from above in $n$. Hence
the spectrum of the Dirichlet Laplacian acting in
$L^2(\Omega_{k,n})$ is discrete. We may assume without loss of
generality that for all $n\in \N,$ $\lambda_k(\Omega_{k,n})\le
2I_k(c)$. Hence
\begin{equation}\label{e12}
\lambda_1(\Omega_{k,n})\le \lambda_k(\Omega_{k,n})\le 2I_k(c).
\end{equation}
It is well-known that for a convex set $\Omega$ in
$\R^m$ the spectrum of the Dirichlet Laplacian acting in
$L^2(\Omega)$ is bounded from below by $(2\rho(\Omega))^{-2}$, see \cite{EBD}. It
follows that
$\lambda_1(\Omega_{k,n})\ge(2\rho(\Omega_{k,n}))^{-2}$. By
\eqref{e12} we conclude that
\begin{equation}\label{e14}
\rho(\Omega_{k,n})\ge 2^{-3/2}I_k(c)^{-1/2}.
\end{equation}
By \eqref{e9}, \eqref{e14} and \eqref{e15a} we have that
$\textup{diam}(\Omega_{k,n})$ is bounded uniformly in $n$ and
satisfies
\begin{align*}
\textup{diam}(\Omega_{k,n})\le
\frac{m2^{(3m-1)/2}}{\omega_{m-1}}\left(\frac{c}{T^*}\right)^{m/{\tau}}I_k(c)^{(m-1)/2}.
\end{align*}
Hence there exists a sequence of translates of $(\Omega_{k,n})$,
again denoted by $(\Omega_{k,n})$, contained in a sufficiently
large closed cube $B_k$. Then $(\overline{\Omega_{k,n}})_{n \in
\N}$ is a sequence of compact sets in $B_k$. The collection of
compact subsets of $B_k$ is compact in the Hausdorff metric by
Theorem 2.4.10 in \cite{HP}. Hence there exists a subsequence,
again denoted by $(\overline{\Omega_{k,n}})_{n \in \N}$ such that
$(\overline{\Omega_{k,n}})$ converges in the Hausdorff metric
to a compact set say $K_k$. Then $K_k$ is convex (Section 2.2 in
\cite{HP}), and by \eqref{e14},
\begin{equation*}
\rho(K_k)\ge 2^{-3/2}I_k(c)^{-1/2}.
\end{equation*}
We conclude that the interior of $K_k$, denoted by $\Omega_k^*$,
is non-empty. Hence $\Omega_{k,n}$ converges to $\Omega_k^*$ in
the Hausdorff metric. Since $\T(\Omega_{k,n})=c,
\rho(\Omega_{k,n})\ge 2^{-3/2}I_k(c)^{-1/2}$ we have by
\eqref{a151} with $A=\Omega_{k,n}$ and $B=\Omega_k^*$ that
$\T(\Omega_k^*)=c$. Furthermore
\begin{align}\label{a159}
|\lambda_k(\Omega_k^*)-I_k(c)|\le
|\lambda_k(\Omega_{k,n})-I_k(c)|+|\lambda_k(\Omega_{k,n})-\lambda_k(\Omega_k^*)|.
\end{align}
The first term in the right hand side of \eqref{a159} tends to $0$
as $n\rightarrow \infty$, since $(\Omega_{k,n})$ is a minimizing
sequence. To estimate the second term in that right hand side we
use \eqref{a152} with $A=\Omega_{k,n}$ and $B=\Omega_k^*$. Since
$\rho(\Omega_{k,n})\ge 2^{-3/2}I_k(c)^{-1/2}$ and
$\lambda_k(\Omega_{k,n})\le 2I_k(c)$ this term tends to $0$ too as
$n\rightarrow \infty$. Hence $\lambda_k(\Omega_k^*)=I_k(c)$ and
$\Omega_k^*$ is a minimizer. This proves \ref{the2}(i).

Since, by \eqref{e14}, the inradius is uniformly bounded from
below and since all elements of the minimizing sequence are convex
the convergence is also in the complementary Hausdorff metric.

To prove part (ii) of Theorem \ref{the2} we consider the set $D$
defined by \eqref{e6a}, and choose $\alpha_c$ such that
$\T(\alpha_c D)=c$. By scaling we have that
\begin{equation}\label{e10}
\alpha_c=\left(\frac{c}{\T(D)}\right)^{1/{\tau}}.
\end{equation}
Hence
\begin{equation}\label{e11}
I_k(c)\le
\lambda_k(\alpha_cD)=\alpha_c^{-2}\lambda_k(D)=\left(\frac{\T(D)}{c}\right)^{2/{\tau}}\lambda_k(D).
\end{equation}
Furthermore by Corollary 1 in \cite{LY}, we have that for any open
set $\Omega$ in $\R^m$ with finite Lebesgue measure,
\begin{equation}\label{e19}
\lambda_k(\Omega)\ge \frac{mC_m}{m+2}
\left(\frac{k}{|\Omega|}\right)^{2/m},
\end{equation}
where
\begin{equation*}
C_m=4\pi^2\omega_m^{-2/m}.
\end{equation*}
It follows by \eqref{e19} and \eqref{e11} that
\begin{align}\label{e21}
|\Omega_k^*|&\ge
\left(\frac{mC_m}{m+2}\right)^{m/2}\frac{k}{\lambda_k(\Omega_k^*)^{m/2}}\nonumber
\\ &\ge\left(\frac{mC_m}{m+2}\right)^{m/2}\frac{k}{I_k(c)^{m/2}}\nonumber \\&\ge\left(\frac{mC_m}{m+2}\right)^{m/2}
\left(\frac{c}{\T(D)}\right)^{m/{\tau}}\frac{k}{\lambda_k(D)^{m/2}}.
\end{align}
By Weyl's Theorem (see e.g. Theorem 10.6 in \cite{S}) we have that
\begin{equation}\label{e28}
\lambda_k(D)=C_mk^{2/m}+o(k^{2/m}),\ k\rightarrow \infty.
\end{equation}
We conclude by \eqref{e21} and \eqref{e28} that
\begin{equation}\label{e29}
\liminf_{k\rightarrow
\infty}|\Omega_k^*|\ge\left(\frac{m}{m+2}\right)^{m/2}\left(\frac{c}{\T(D)}\right)^{m/{\tau}}.
\end{equation}
Hence by \eqref{e7} and \eqref{e29} we have that
\begin{equation}\label{e30}
\limsup_{k\rightarrow\infty}\textup{diam}(\Omega_k^*)\le K
\left(\frac{m+2}{m}\right)^{(t\tau-1)/2}c^{1/{\tau}}\T(D)^{(t\tau-1)/{\tau}}.
\end{equation}
Hence there exists a sufficiently large cube $B$ which contains
translates of $\Omega_k^*,k=1,2,\cdots$ again denoted by
$\Omega_k^*,k=1,2,\cdots$. As before $(\overline{\Omega_k^*})_{k
\in \N}$ is a sequence of compact sets in $B$. The collection of
compact subsets of $B$ is compact, in the Hausdorff metric by
Theorem 2.4.10 in \cite{HP}. Hence there exists a convergent
subsequence $(\overline{\Omega_{k_l}^*})_{l\in \N}$ which
converges to a convex, compact set $K$. By \eqref{e15b},\
\eqref{e29} and \eqref{e30} we have that
\begin{align}\label{e25a}
\rho(K)&=\lim_{l\rightarrow
\infty}\rho(\Omega_{k_l}^*)\ge\liminf_{k\rightarrow\infty}\rho(\Omega_k^*)\nonumber
\\ &\ge
2^{m-1}K^{1-m}(m\omega_m)^{-1}\left(\frac{m}{m+2}\right)^{(t\tau(m-1)+1)/2}c^{1/\tau}\T(D)^{(t\tau(1-m)-1)/\tau}.
\end{align}
Then the interior of $K$, denoted by $\Omega^*,$  is non-empty. We
now use \eqref{a151} with $A=\Omega_{k_l}^*$ and $B=\Omega^*$ to
conclude that $\T(\Omega^*)=c$. Let $ \epsilon>0$ be arbitrary.
There exists $l_0\in\N$ such that $l\ge l_0$ implies
$\Omega_{k_l}^*\subset {\Omega^*}^{\epsilon}$. By monotonicity of
Dirichlet eigenvalues $\lambda_{k_l}({\Omega^*}^{\epsilon})\le
\lambda_{k_l}(\Omega_{k_l}^*)\le\lambda_{k_l}(\alpha_cD)$. It
follows that, by applying Weyl's Theorem to both sides of the
inequality above, $|{\Omega^*}^{\epsilon}|\ge
|\alpha_cD|=\alpha_c^m$. Since $\epsilon>0$ was arbitrary we
conclude that $|\Omega^*|\ge |\alpha_c D|=\alpha_c^m$. On the
other hand, by \eqref{e8} and \eqref{e10} we have that
\begin{equation*}
|\Omega^*|\le \left(\frac{\T(\Omega^*)}{\T(D)}\right)^{m/{\tau}}=\left(\frac{c}{\T(D)}\right)^{m/{\tau}}=\alpha_c^m.
\end{equation*}
By uniqueness of the set $D$ in \eqref{e6} the only set which
satisfies the constraint in \eqref{e5} and has measure
$\alpha_c^m$ is an isometry of $\alpha_cD$. We conclude that the
subsequence $\Omega_{k_l}^*$ converges in the Hausdorff metric to
an isometry of $\alpha_cD$. This concludes the proof of
\eqref{e6b} since the limit is independent of the subsequence.

The convergence also takes place in the complementary Hausdorff
metric since the elements in the sequence are convex, and by
\eqref{e6b}, their inradii are uniformly bounded from below. This
concludes the proof Theorem \ref{the2}(ii).

\begin{corollary}\label{cor3} Let $\Omega_k^*$ be a minimizer of
\begin{equation*}
I_k(c)=\inf \{\lambda_k(\Omega) :\Omega\ \textup{open, convex in}\
\R^m ,\Pa(\Omega) = c \},
\end{equation*}
then there exists a sequence of translates of $(\Omega_k^*)$,
again denoted by $(\Omega_k^*),$ which converges to the ball with
perimeter $c$ in the Hausdorff metric.
\end{corollary}
\begin{proof}
It follows by the isoperimetric inequality that if $\Pa(\Omega)=c$
then $|\Omega|<\infty$, and so the Dirichlet Laplacian acting in
$L^2(\Omega)$ has discrete spectrum. The scaling relation under
(b) holds with $\tau=m-1$. By \cite{GWW} we have that the diameter
bound under (d) holds with $t=m-1$, and reads
\begin{equation}\label{e37}
\textup{diam}(\Omega)\le
m^{2-m}\omega_{m-1}^{-1}\Pa(\Omega)^{m-1}|\Omega|^{2-m}.
\end{equation}
The unique set $D$ under \eqref{e6a} is the ball in $\R^m$ with
Lebesgue measure $1$.
\end{proof}
Note that the constant in \eqref{e37} is sharp for a sequence of
double sided cones with diameter increasing to infinity
\cite{GWW}.

Recall that the moment of inertia of an open set $\Omega$ in
$\R^m$ with respect to its centre of mass is defined by
\begin{equation}\label{e48}
\mathcal{J}(\Omega)= \frac{1}{2|\Omega|}\iint_{\Omega\times\Omega}dxdy|x-y|^2.
\end{equation}
\begin{corollary}\label{cor4} Let $\Omega_k^*$ be a minimizer of
\begin{equation*}
I_k(c)=\inf \{\lambda_k(\Omega) :\Omega\ \textup{open, convex in}\
\R^m ,\mathcal J(\Omega) = c \},
\end{equation*}
Then there exists a sequence of translates of
$(\Omega_k^*)$, again denoted by $(\Omega_k^*)$ which converges in the Hausdorff metric to
the ball with moment of inertia $c$.
\end{corollary}
\begin{proof}
From \eqref{e48} it is clear that the moment of inertia is
invariant under isometries and monotone on the open sets. Hence
(a) is satisfied. By \eqref{e48} we see that the scaling under (b)
holds with $\tau=m+2$. The isoperimetric inequality for the moment
of inertia states that
\begin{equation}\label{e39}
\mathcal{J}(\Omega)\ge
\frac{m}{m+2}\omega_m^{-2/m}|\Omega|^{(m+2)/m},
\end{equation}
with equality if and only if $\Omega$ is a ball (up to sets of
measure $0$). The isoperimetric inequality \eqref{e39} implies
that (c2) holds for the ball with Lebesgue measure $1$. Below we
show that the diameter bound under (d) holds with $t=\frac{1}{2}$,
and reads
\begin{equation}\label{e41}
\textup{diam}(\Omega)\le K
\mathcal{J}(\Omega)^{1/2}|\Omega|^{-1/2},
\end{equation}
where
\begin{equation}\label{e42}
K=4\left(m(m+1)^2(m+2)\right)^{1/2}.
\end{equation}
Note that by \eqref{e15a} an open, convex set with finite Lebesgue
measure is bounded. To prove \eqref{e41} we let $d$ and $f$ be two
points of the boundary of $\Omega$ such that
$|d-f|=\textup{diam}(\Omega)$. Let $p$ be any point of the
straight line segment $[d,f]$, and let $\Pi_p$ be the
$(m-1)$-dimensional plane perpendicular to $[d,f]$. We denote
$\Omega_p=\Pi_p\cap\Omega$, and denote its $(m-1)$-dimensional
Lebesgue measure by $|\Omega_p|_{m-1}$. Then $p\mapsto
|\Omega_p|_{m-1}$ is continuous on the compact line segment
$[d,f]$, and attains its maximum at say $0$. We choose the
positive $x_1$ axis along $[0,f]$, and put $x=(x_1,x')$. Without
loss of generality we may assume that $|f|\ge
\frac{1}{2}\textup{diam}(\Omega)$. By \eqref{e48} we have that
\begin{align}\label{e43}
\mathcal{J}(\Omega)&\ge\frac{1}{2|\Omega|}\int_0^{|f|}dx_1\int_0^{|f|}dy_1\int_{\Omega_{x_1}}dx'\int_{\Omega_{y_1}}dy'\left((x_1-y_1)^2+|x'-y'|^2\right)
\nonumber \\
&\ge\frac{1}{2|\Omega|}\int_0^{|f|}dx_1\int_0^{|f|}dy_1\int_{\Omega_{x_1}}dx'\int_{\Omega_{y_1}}dy'(x_1-y_1)^2
\nonumber \\
&=\frac{1}{2|\Omega|}
\int_0^{|f|}dx_1\int_0^{|f|}dy_1|\Omega_{x_1}|_{m-1}|\Omega_{y_1}|_{m-1}(x_1-y_1)^2.
\end{align}
By convexity we have that $\Omega$ contains the cone with base
$\Omega_{x_1}$ and vertex $f$. It follows by monotonicity and
scaling that
\begin{equation}\label{e44}
|\Omega_{x_1}|_{m-1}\ge|\Omega_{0}|_{m-1}\left(1-\frac{x_1}{|f|}\right)^{m-1}.
\end{equation}
By \eqref{e43} and \eqref{e44} we find that
\begin{align*}
\mathcal{J}(\Omega)&\ge\frac{|\Omega_{0}|_{m-1}^2}{2|\Omega|}\int_0^{|f|}dx_1\int_0^{|f|}dy_1\left(1-\frac{x_1}{|f|}\right)^{m-1}
\left(1-\frac{y_1}{|f|}\right)^{m-1}(x_1-y_1)^2
\nonumber \\
&=\frac{|\Omega_{0}|_{m-1}^2|f|^4}{2|\Omega|}\int_0^1dx_1\int_0^1dy_1(1-x_1)^{m-1}
(1-y_1)^{m-1}(x_1-y_1)^2\nonumber\\
&=\frac{|\Omega_{0}|_{m-1}^2|f|^4}{m(m+1)^2(m+2)|\Omega|}\nonumber\\
&\ge\frac{|\Omega_{0}|_{m-1}^2\textup{diam}(\Omega)^4}{16m(m+1)^2(m+2)|\Omega|}\nonumber\\
&\ge\frac{\textup{diam}(\Omega)^2|\Omega|}{16m(m+1)^2(m+2)},
\end{align*}
where we have used that $|\Omega|\le
|\Omega_0|_{m-1}\textup{diam}(\Omega)$. This implies \eqref{e41},
\eqref{e42}.
\end{proof}
\begin{corollary}\label{cor5}
If $\T$ satisfies \textup{(a),(b) and (c1)} then
\begin{equation}\label{e52}
H_k=\inf \{\lambda_k(\Omega)+\T(\Omega) :\Omega\ \textup{open, convex in}\ \R^m \},
\end{equation}
has a minimizer which is up to isometries a homothety of the minimizer of \eqref{e5} with $c=1$.
\end{corollary}
The proof of this corollary is straightforward and is deferred to
the Appendix.

\section{Proof of Theorem \ref{the3} \label{sec3}}

In this section we will prove Theorem \ref{the3}.

\noindent {\begin{proof}(i) It is evident from the definition of
$J_k(c)$ that $c\mapsto J_k(c)$ is monotonically decreasing. To
prove continuity we let $c_2>0$ and we let $\epsilon>0$ be
arbitrary. Then there exists an open set $\Omega_{\epsilon}$ such
that $\lambda_k(\Omega_{\epsilon})\le (1+\epsilon)J_k(c_2)$, and
which satisfies the constraints $|\Omega_{\epsilon}|\le 1$ and
$\Pa(\Omega_{\epsilon})\le c_2$. Let $c_1<c_2$. Then the open set
$\left(\frac{c_1}{c_2}\right)^{1/(m-1)}\Omega_{\epsilon}$
satisfies
\begin{equation*}
\left|(\frac{c_1}{c_2})^{1/(m-1)}\Omega_{\epsilon}\right|\le
\left(\frac{c_1}{c_2}\right)^{m/(m-1)}<1,
\end{equation*}
and
\begin{equation*}
\Pa\left((\frac{c_1}{c_2})^{1/(m-1)}\Omega_{\epsilon}\right)\le
c_1.
\end{equation*}
Hence
\begin{equation*}
J_k(c_1)\le\lambda_k\left((\frac{c_1}{c_2})^{1/(m-1)}\Omega_{\epsilon}\right)\le(1+\epsilon)\left(\frac{c_2}{c_1}\right)^{2/(m-1)}J_k(c_2).
\end{equation*}
This implies \eqref{b4} since $\epsilon>0$ was arbitrary. To prove
continuity we have for $c_1<c_2$ by \eqref{b4} and the
monotonicity of $c\mapsto J_k(c)$ that
\begin{equation*}
0\le J_k(c_1)-J_k(c_2)\le
J_k(c_2)\left((\frac{c_2}{c_1})^{2/(m-1)}-1\right).
\end{equation*}
This implies left-continuity at $c_2$. We have by \eqref{b4} and
monotonicity of $c\mapsto J_k(c)$ that
\begin{equation*}
0\le J_k(c_1)-J_k(c_2)\le
J_k(c_1)\left(1-(\frac{c_1}{c_2})^{2/(m-1)}\right).
\end{equation*}
This implies right-continuity at $c_1$.

\noindent(ii) Suppose that $c>\mu_k$. By the definition of $\mu_k$
there exists $\Omega^*\in \mathfrak{M}_k$ with $\Pa(\Omega^*)\le
c$. Hence $J_k(c)\le \lambda_k(\Omega^*)=M_k(1)$. Since trivially
$J_k(c)\ge M_k(1)$ we have that $J_k(c)=M_k(1)$ for $c>\mu_k$, and
that $\Omega^*$ is a minimizer of \eqref{b1}. Finally
$J_k(\mu_k)=M_k(1)$ by the continuity of $c\mapsto J_k(c)$ proved
above.

\noindent(iii)
Suppose that $c<\pi_k^{-(m-1)/m}$. Let $\mathfrak{P}_k(c)$ denote
the collection of minimizers of $P_k(c)$, and put
\begin{equation*}
\pi_k(c)=\inf\{|\Omega|:\ \Omega\in \mathfrak{P}_k(c)\}.
\end{equation*}
By scaling we have that
\begin{align*}
\pi_k(c)&=\inf\{|c^{1/(m-1)}\Omega|,\ c^{1/(m-1)}\Omega\in
\mathfrak{P}_k(c)\}\nonumber \\ &= c^{m/(m-1)}\inf\{|\Omega|:\
\Omega\in \mathfrak{P}_k(1)\}\nonumber \\ &=c^{m/(m-1)}\pi_k.
\end{align*}
First suppose that $\pi_k(c)<1$. Then there exists a minimizer
$\tilde{\Omega}\in\mathfrak{P}_k(c)$ with $|\tilde{\Omega}|\le1$.
Hence $J_k(c)\le \lambda_k(\tilde{\Omega})\le P_k(c)$. On the
other hand, $J_k(c)\ge P_k(c)$. It follows that $J_k(c)= P_k(c)$,
$\Omega_*:=c^{-1/(m-1)}\tilde{\Omega}\in\mathfrak{P}_k$, and that
$c^{1/(m-1)}\Omega_*$ is a minimizer of \eqref{b1}. We have that
$J_k(c)= P_k(c)$ if $\pi_k(c)=1$ by the continuity of $c\mapsto
J_k(c)$ proved above.

\noindent(iv) Let $\Omega^* \in \mathfrak{M}_k$. Then
$|\Omega^*|=1$, and $\Pa(\Omega^*)\ge m\omega_m^{1/m}$ by the
isoperimetric inequality.

\noindent(v) Let $\Omega_*\in \mathfrak{P}_k$.
Then $\Pa(\Omega)=1$, and $|\Omega|\le
m^{-m/(m-1)}\omega_m^{-1/(m-1)}$ by the isoperimetric inequality.

\noindent(vi) By the Li-Yau inequality \eqref{e19},
\begin{equation}\label{e60}
P_k(1)=\lambda_k(\Omega^*)\ge \frac{mC_m}{m+2}
\left(\frac{k}{|\Omega^*|}\right)^{2/m}.
\end{equation}
For the cube $Q_a \in \R^m$ with $|Q_a|=a^m$ and with $\Pa(Q_a)=1$
we have that
\begin{equation}\label{e61}
a=(2m)^{-1/(m-1)}.
\end{equation}
We have that $\lambda_k(Q_a)\ge \lambda_1(Q_a)=m\pi^2/a^2$. Since
for $x> 1$ we have that $\max\{n\in\N:n<x\}=\lfloor x\rfloor\ge
x/2$, we conclude that
\begin{align*}
k&=\sharp\{(k_1,\cdots,k_m)\in\N^m:\pi^2(k_1^2+\cdots
+k_m^2)\le\lambda_k(Q_a)a^2\}\nonumber \\ &\ge \left(\sharp\{k\in
\N:m\pi^2k^2<\lambda_k(Q_a)a^2\}\right)^m\nonumber \\ &\ge \left
\lfloor\frac{a\lambda_k(Q_a)^{1/2}}{\pi m^{1/2}} \right\rfloor^m\nonumber \\
&\ge\left(\frac{a\lambda_k(Q_a)^{1/2}}{2\pi m^{1/2}} \right)^m.
\end{align*}
It follows that
\begin{equation}\label{e63}
\lambda_k(Q_a)\le \frac{4\pi^2m}{a^2}k^{2/m}.
\end{equation}
We also have that $P_k(1)\le \lambda_k(Q_a)$. Putting this
together with \eqref{e60}, \eqref{e61} and \eqref{e63} we conclude
that
\begin{equation*}
|\Omega^*|\ge (2m)^{-m/(m-1)}(m+2)^{-m/2}\omega_m^{-1}.
\end{equation*}
\end{proof}

\section{Appendix\label{sec4}}
In this appendix we prove the following.
\begin{proposition}\label{prop1}Let
\begin{equation}\label{a1}
L_k=\inf \{\lambda_k(\Omega) :\Omega \ \textup{open in}\ \R^m ,\
\T(\Omega) \le 1 \},
\end{equation}
and
\begin{equation}\label{a2}
T_k=\inf \{\lambda_k(\Omega)+\T(\Omega) :\Omega\ \textup{open in}\
\R^m \}.
\end{equation}
Suppose $\T$ satisfies the hypotheses \textup{(a), (b)} and
\textup{(c1)}. Then the variational problem defined under
\eqref{a1} has a minimizer if and only if the variational problem
under \eqref{a2} has a minimizer. Moreover these minimizers are
hometheties of one another.
\end{proposition}
\begin{proof}
It is convenient to define
\begin{equation}\label{a3}
N_k=\inf \{\lambda_k(\Omega)\mathcal{T}(\Omega)^{2/{\tau}}
:\Omega\ \textup{open in}\ \R^m ,\ \T(\Omega)<\infty \}.
\end{equation}
First we show that $L_k=N_k$. We have by scaling of the Dirichlet
eigenvalues  that
\begin{align}\label{a4} L_k=&
 \inf \{\lambda_k(t\Omega) :\Omega \ \textup{open in}\ \R^m ,\
t>0,\ \T(t\Omega) \le 1 \}\nonumber \\
=&\inf\{t^{-2}\lambda_k(\Omega):\Omega \ \textup{open in}\ \R^m ,\
t>0,\ t^{\tau}\T(\Omega)\le 1\}\nonumber \\
\ge &\inf\{\T(\Omega)^{2/{\tau}}\lambda_k(\Omega):\Omega \
\textup{open in}\ \R^m ,\ t>0,\ t^{\tau}\T(\Omega)\le 1\}\nonumber
\\
\ge&\inf\{\T(\Omega)^{2/{\tau}}\lambda_k(\Omega):\Omega \
\textup{open in}\ \R^m, \T(\Omega)<\infty \}\nonumber \\ =&N_k.
\end{align}
We obtain the reverse inequality by choosing
$t=\T(\Omega)^{-1/{\tau}}$ in the second line of \eqref{a4}. If
$\Omega^*$ is a minimizer of \eqref{a1} then $\T(\Omega^*)=1$.
Hence
$\T(\Omega^*)^{2/{\tau}}\lambda_k(\Omega^*)=\lambda_k(\Omega^*)=L_k=N_k$,
and the infimum in \eqref{a3} is attained for $\Omega^*$.
Conversely if $\Omega^*$ is a minimizer of \eqref{a3} then we
choose $\alpha>0$ such that $\T(\alpha\Omega^*)=1$. Hence
$\alpha=\T(\Omega^*)^{-1/{\tau}}$. So $\alpha\Omega^*$ satisfies
the constraint in \eqref{a1}, and
$\lambda_k(\alpha\Omega^*)=\alpha^{-2}\lambda_k(\Omega^*)=\T(\Omega^*)^{2/{\tau}}\lambda_k(\Omega^*)=N_k=L_k$.
Hence the infimum in \eqref{a1} is attained by a homothety of
$\Omega^*$. We conclude that \eqref{a1} has a minimizer if and
only if \eqref{a3} has a minimizer.

Next we show that the variational problem under \eqref{a2} has a
minimizer if and only if the variational problem under \eqref{a3}
has a minimizer. We note that
\begin{align*}
T_k=&\inf \{\lambda_k(\Omega)+\T(\Omega) :\Omega\ \textup{open
in}\ \R^m \}\nonumber \\  =&\inf \{\lambda_k(t\Omega)+\T(t\Omega)
:\Omega\ \textup{open in}\ \R^m  ,\ t>0\}\nonumber \\  =&\inf
\{t^{-2}\lambda_k(\Omega)+t^{\tau}\T(\Omega) :\Omega\ \textup{open
in}\ \R^m  ,\ t>0\ , \T(\Omega)<\infty \}.
\end{align*}
The infimum over $t>0$ is attained for
\begin{equation}\label{a6}
t(\Omega)=\left(\frac{2\lambda_k(\Omega)}{\tau\T(\Omega)}\right)^{1/(\tau+2)}.
\end{equation}
Hence
\begin{align}\label{a7}
T_k=n(\tau)\inf\{\lambda_k(\Omega)^{\tau/(\tau+2)}\T(\Omega)^{2/(\tau+2)}:\Omega\
\textup{open in}\ \R^m \} =n(\tau)N_k^{\tau/(\tau +2)},
\end{align}
where
\begin{equation*}
n(\tau)=\left(\frac{\tau}{2}\right)^{2/(\tau+2)}+\left(\frac{2}{\tau}\right)^{\tau/(\tau
+2)}.
\end{equation*}
If $\Omega^*$ is a minimizer of \eqref{a3} then by \eqref{a2},
\eqref{a6} and \eqref{a7},
\begin{align*}
T_k &\le
\lambda_k(t(\Omega^*)\Omega^*)+\T(t(\Omega^*)\Omega^*)\nonumber \\
&=t(\Omega^*)^{-2}\lambda_k(\Omega^*)+t(\Omega^*)^{\tau}\T(\Omega^*)\nonumber
\\ &=n(\tau)N_k^{\tau/(\tau +2)}\nonumber \\ &=T_k.
\end{align*}
Hence $t(\Omega^*)\Omega^*$ is a minimizer of \eqref{a2}. If
$\Omega^*$ is a minimizer of \eqref{a2} then by \eqref{a7} we have
that
\begin{align*}
N_k^{\tau/(\tau+2)}&\le(\lambda_k(\Omega^*)\mathcal{T}(\Omega^*)^{2/\tau})^{\tau/(\tau+2)}\nonumber
\\ &=n(\tau)^{-1}\inf
\{t^{-2}\lambda_k(\Omega^*)+t^{\tau}\T(\Omega^*):\ t>0\}\nonumber
\\ &=n(\tau)^{-1}\inf \{\lambda_k(t\Omega^*)+\T(t\Omega^*):\
t>0\}\nonumber \\
&\le n(\tau)^{-1}(\lambda_k(\Omega^*)+\T(\Omega^*))\nonumber \\
&=n(\tau)^{-1}T_k\nonumber \\ &=N_k^{\tau/(\tau+2)}.
\end{align*}
Hence $\Omega^*$ is a minimizer of \eqref{a3}. We conclude that
\eqref{a2} has a minimizer if and only if \eqref{a3} has a
minimizer. This also concludes the proof of the proposition.
\end{proof}
 \noindent {\it Proof of Corollary
\ref{cor5}.} We remark that Proposition \ref{prop1} holds if the
variational expressions under \eqref{a1} and \eqref{a2} have an
additional convexity constraint. If $\T$ satisfies \textup{(a),
(b) and (c1)} then \eqref{e5} has a minimizer. By the previous
remark we have that Proposition \ref{prop1} implies that the
variational expression under \eqref{e52} has a minimizer which is
a homothety of the one corresponding to \eqref{e5}.\hspace*{\fill
}$\square $

 \noindent

\end{document}